\documentclass[leqno,11pt]{scrartcl}
\usepackage[utf8]{inputenc}

\usepackage{amsmath,amssymb,amsthm}
\usepackage{amscd}
\usepackage{setspace} 
\usepackage{enumerate}
\usepackage{array}
\usepackage{tabularx}
\usepackage{multirow}
\usepackage{booktabs}

\theoremstyle{definition}
\newtheorem{definition}{Definition}
\newtheorem{example}{Example}
\newtheorem{remark}{Remark}

\theoremstyle{plain}
\newtheorem{theorem}{Theorem}

\newtheorem{lemma}[definition]{Lemma}
\newtheorem{corollary}{Corollary}

\theoremstyle{remark}

\newcommand{\K}{\mathbb{K}}
\newcommand{\N}{\mathbb{N}}

\newcommand{\Q}{\mathbb{Q}}
\newcommand{\R}{\mathbb{R}}
\newcommand{\Z}{\mathbb{Z}}

\newcommand{\Ccal}{\mathcal{C}}

\newcommand{\disc}{\operatorname{disc}}

\newcommand{\oh}{{\scriptstyle{{\cal O}}}}

\let\leq\leqslant

\begin{document}

\begin{center}
\begin{huge}
\begin{spacing}{1.0}
\textbf{Maximal Discrete Subgroups of $SO^+(2,n+2)$}  
\end{spacing}
\end{huge}

\bigskip
by
\bigskip

\begin{large}
\textbf{Aloys Krieg\footnote{Aloys Krieg, Lehrstuhl A für Mathematik, RWTH Aachen University, D-52056 Aachen, Germany, krieg@rwth-aachen.de}}
 and
\textbf{Felix Schaps\footnote{Felix Schaps, Lehrstuhl A für Mathematik, RWTH Aachen University, D-52056 Aachen, Germany, felix.schaps@matha.rwth-aachen.de}}
\end{large}
\vspace{0.5cm}\\
September 2021
\vspace{1cm}
\end{center}
\begin{abstract}
We characterize the maximal discrete subgroups of $SO^+(2,n+2)$, which contain the discriminant kernel of an even lattice, which contains two hyperbolic planes over $\Z$.  They coincide with the normalizers in $SO^+(2,n+2)$ and are given by the group of all integral matrices inside $SO^+(2,n+2)$, whenever the underlying lattice is maximal even. Finally we deal with the irreducible root lattices as examples.
\end{abstract}
\noindent\textbf{Classification: 11F06, 11F55}
\vspace{2ex}\\

\newpage
\section{Introduction}

The Hermitian symmetric space associated with the special orthogonal group $SO(2,n+2)$ is a Siegel domain of type IV. The attached spaces of modular forms have attracted a lot of attention, mainly influenced by the product expansions of Borcherds (cf. \cite{Bo2}). Recently a lot of concrete examples for small $n$ have been constructed by Wang and Williams (cf. \cite{Wa21} - \cite{WW4}). The modular group consists of the discriminant kernel of an even lattice as well as certain congruence subgroups (cf. \cite{H-WK}).

Moreover the Maaß lift or additive lift has been described by Gritsenko (cf. \cite{G4}, \cite{K8}). In a recent paper by Wernz \cite{We2} the connection between different notions of Maaß spaces for $SO(2,4)$ has been reduced to modular forms for the discriminant kernel versus its maximal discrete extension.

In this paper we consider the case of general $n$. We determine the maximal discrete extension of the discriminant kernel and show that it is equal to the group of all integral matrices inside $SO^+(2,n+2)$, whenever we start with a maximal even lattice with two hyperbolic planes over $\Z$. In this case it also coincides with the normalizer. To a certain extent this characterizes the maximal even lattices among all even lattices. 

\section{Maximal Even Lattices}

We start with an even lattice $\Lambda$ in a $\Q$-vector space $V$ of dimension $n$ equipped with a non-degenerate symmetric bilinear form $\langle \cdot,\cdot \rangle$, i.e. $\Lambda$ is a free group of rank $n$ satisfying $\langle \lambda,\lambda\rangle \in 2\Z$ for all $\lambda \in \Lambda$. The dual lattice is given by 
\[
 \Lambda^\sharp:= \{v\in V;\; \langle v,\Lambda\rangle \subseteq \Z\} \supseteq \Lambda
\]
and $\Lambda^\sharp/\Lambda$ with the quadratic form
\[
 \overline{q}: \Lambda^\sharp/\Lambda \to \Q/\Z,\quad \lambda + \Lambda \mapsto \tfrac{1}{2} \langle \lambda,\lambda\rangle + \Z,
\]
is called the \emph{discriminant group} of $\Lambda$. The lattice $\Lambda$ is always contained in a maximal even lattice in $V$, which is a sublattice of $\Lambda^\sharp$ (cf. \cite{Kn}, 14.11).

Throughout the paper we choose a basis of a positive definite lattice $L$ with Gram matrix $S$. Let $\disc L:= \det S$ denote its discriminant. We add two hyperbolic planes over $\Z$, i.e.
\begin{gather*}\tag{1}\label{gl_1}
 \begin{cases}
  & L=\Z^n,\quad S\in\Z^{n\times n} \:\text{positive definite and even}, \\[1ex]
  & L_0 = \Z^{n+2}, \quad S_0 = \left(\begin{smallmatrix}
                                 0 & 0 & 1 \\ 0 & -S & 0 \\ 1 & 0 & 0
                                \end{smallmatrix}\right),   \\[2ex]
  & L_1 = \Z^{n+4},\quad S_1 = \left(\begin{smallmatrix}
                                 0 & 0 & 1 \\ 0 & S_0 & 0 \\ 1 & 0 & 0
                                \end{smallmatrix}\right). 
 \end{cases}
\end{gather*}
Thus $S_1$ has got the signature $(2,n+2)$.

We consider the attached special orthogonal group
\[
 SO(S_1;\R):=\{M\in SL(n+4;\R);\;M^{tr} S_1 M=S_1\}.
\]
Let $SO^+(S_1;\R)$ stand for the connected component of the identity matrix $I$. Due to (5) in \cite{K1}, it can be characterized 
by 
\[                                                                                                                                   \det (CP+D) > 0, \quad P= \begin{pmatrix}
                           0 & 1 \\ 1 & 0                                                                                                                                   
			  \end{pmatrix}, \quad                                                                                                                                                            
 M= \begin{pmatrix}
     \ast & \ast & \ast \\ C & \ast & D
    \end{pmatrix} \in SO(S_1;\R)
\]
with $2\times 2$ matrices $C,D$. Given $M\in SO^+(S_1;\R)$ we will always assume the form 
\begin{gather*}\tag{2}\label{gl_2}
 M = \begin{pmatrix}
      \alpha & a^{tr} S_0 & \beta \\ b & K & c \\ \gamma & d^{tr} S_0 & \delta
     \end{pmatrix},\; \alpha, \beta, \gamma, \delta \in\R, \;\; a,b,c,d\in\R^{n+2}, \; K\in \R^{(n+4)\times(n+4)}.
\end{gather*}
Its inverse is given by
\begin{gather*}\tag{3}\label{gl_3}
 M^{-1} = S_1^{-1} M^{tr} S_1 = \begin{pmatrix}
                                 \delta & c^{tr}S_0 & \beta \\ d & S_0^{-1} K^{tr} S_0 & a \\
                                 \gamma & b^{tr} S_0 & \alpha
                                \end{pmatrix}.
\end{gather*}
Let $\Gamma_S:=SO^+(S_1;\Z)$ denote the subgroup of integral matrices. Note that in this case $a,d\in \Z^{n+2}$ holds in \eqref{gl_2} due to $M^{-1}\in\Gamma_S$ and \eqref{gl_3}. Moreover we define the \emph{discriminant kernel} 
\[
 \widetilde{\Gamma}_S := \{M\in \Gamma_S;\; M\in I + \Z^{(n+4)\times (n+4)} S_1\},
\]
where $I$ is the identity matrix. The discriminant kernel induces the identity on $L^\sharp_1/L_1$, $L^\sharp_1 = S^{-1}_1 \Z^{n+4}$. We consider particular matrices in $\widetilde{\Gamma}_S$:
\begin{gather*}\tag{4}\label{gl_4}
 J= \begin{pmatrix}
     0 & 0 & -P \\ 0 & I & 0 \\ -P & 0 & 0
    \end{pmatrix}, \quad
 P= \begin{pmatrix}
     0 & 1 \\ 1 & 0
    \end{pmatrix},
\end{gather*}
\begin{gather*}\tag{5}\label{gl_5}
 T_\lambda = \begin{pmatrix}
              1 & -\lambda^{tr} S_0 & -\tfrac{1}{2} \lambda^{tr} S_0 \lambda \\ 0 & I & \lambda \\ 0 & 0 & 1 
             \end{pmatrix}, \quad
 T^*_\lambda = \begin{pmatrix}
               1 & 0 & 0 \\ \lambda & I & 0 \\ -\tfrac{1}{2} \lambda^{tr} S_0 \lambda & -\lambda^{tr} S_0 & 1
              \end{pmatrix},\;\;
 \lambda \in \Z^{n+2}.
\end{gather*}
At first we give a description of the first columns of matrices in $\Gamma_S$.

\begin{theorem}\label{theorem_1} 
 Let $L_1=\Z^{n+4}$ satisfy \eqref{gl_1}. Given $h\in\Z^{n+4}$ the following assertions are equivalent:
 \begin{enumerate}
  \item[(i)] $h$ is the first column of a matrix in $\Gamma_S$ (resp. $\widetilde{\Gamma}_S$).
  \item[(ii)] $h^{tr}S_1$ is the last row of a matrix in $\Gamma_S$ (resp. $\widetilde{\Gamma}_S$).
  \item[(iii)] $h^{tr} S_1 h=0$ and $\gcd(S_1 h)=1$.
 \end{enumerate}
\end{theorem}

\begin{proof}
 (i) $\Leftrightarrow$ (ii) ~Use \eqref{gl_2} and \eqref{gl_3}. \\[1ex]
 (i) $\Rightarrow$ (iii) ~Apply $M^{tr} S_1 M= S_1$ and $\widetilde{\Gamma}_S\subseteq \Gamma_S \subseteq SL(n+4;\Z)$.  \\[1ex]
 (iii) $\Rightarrow$ (ii) ~Proceed in the same way as in the proof of Theorem \ref{theorem_1} in \cite{K1}. The matrices involved there lead to an $M\in \widetilde{\Gamma}_S$ such that
 \[
  h^{tr} S_1 M = (0,\ldots,0,1).
 \]
Hence $h^{tr} S_1$ is the last row of $M^{-1} \in \widetilde{\Gamma}_S \subseteq \Gamma_S$.
\end{proof}

In the context of the action of $\Gamma_S$ on the orthogonal half-space (cf. \cite{G4}), it makes sense to consider cusps. If $\Gamma$ is a subgroup of $\Gamma_S$ of finite index, we denote by
\[
 \Ccal^0(\Gamma):=\bigl\{\Gamma h;\; h\in L^\sharp_1 = S^{-1}_1 \Z^{n+4},\,h^{tr} S_1 h=0,\, \gcd(S_1 h) = 1\bigr\}
\]
the set of $\Gamma$-orbits of \emph{zero-dimensional cusps} (cf. \cite{GHS}).

\begin{corollary}\label{corollary_1} 
 Let $L_1 = \Z^{n+4}$ satisfy \eqref{gl_1}. Then the following assertions are equivalent:
 \begin{enumerate}
  \item[(i)] $L=\Z^n$ is maximal even.
  \item[(ii)] Every $g\in L^\sharp_1 = S^{-1}_1 \Z^{n+4}$ with $g^{tr} S_1 g=0$ fulfills $g\in L_1 = \Z^{n+4}$.
  \item[(iii)] $\sharp \Ccal^0(\Gamma_S) = 1$. 
  \item[(iv)] $\sharp \Ccal^0(\widetilde{\Gamma}_S) = 1$. 
 \end{enumerate}
\end{corollary}

\begin{proof}
 (i) $\Rightarrow$ (iv) ~Let
 \[
  g= (g_1,g_2,\lambda_1,\ldots,\lambda_n,g_3,g_4)^{tr}\in L^\sharp_1, g^{tr} S_1 g = 0,\;\; \gcd(S_1 g) = 1.
 \]
Thus $\lambda = (\lambda_1,\ldots,\lambda_n)^{tr}\in L^\sharp$ follows with
\[
 \lambda^{tr} S\lambda = 2(g_1 g_4 + g_2 g_3) \in 2\Z.
\]
Hence $L+\Z \lambda$ is an even overlattice of $L$ and (i) yields $\lambda \in L$, i.e. $g\in L_1$. Then Theorem \ref{theorem_1} leads to (iv). \\[1ex]
(iv) $\Rightarrow$ (iii) ~This is clear due to $\widetilde{\Gamma}_S \subseteq \Gamma_S$. \\[1ex]
(iii) $\Rightarrow$ (ii) ~$\Gamma_S$ acts transitively on the set of vectors $g\in L^\sharp_1$ with $g^{tr} S_1 g = 0$ and $\gcd(S_1 g)=1$. In view of $\Gamma_S\subseteq SL(n+4;\Z)$ any $M\in \Gamma_S$ induces a bijective map  $M:\Z^{n+4} \to \Z^{n+4}$, $h\mapsto Mh$. Hence $L^\sharp_1 = L_1$ follows. \\[1ex]  
(ii) $\Rightarrow$ (i) ~This is clear as any even overlattice of $L$ is contained in $L^\sharp$ and 
\[
 \{ \lambda \in L^\#; \lambda^{tr} S \lambda \in 2 \Z \} =  L
\]
Hence $L$ is maximal even.
\end{proof}
Corollary \ref{corollary_1} says that $L$ is maximal even if and only if $(L^\sharp/L,\overline{q})$ is anisotropic. The equivalence between (i) and (ii) is contained in \cite{Ni}, Proposition 1.4.1, under weaker assumptions.
\vspace{1ex}

We give some examples.

\begin{example}\label{example_1} 
 a) Considering $L=\Z$ with $\langle x,y\rangle = 2Nxy$, $N\in \N$, we obtain a maximal even lattice if and only if $N$ is squarefree. This leads to paramodular groups (cf. \cite{GKr}).  \\[1ex]
 b) If $L=\oh_\K$ is the ring of integers of an imaginary quadratic number field $\K$ with $\langle x,y\rangle = x\overline{y} + \overline{x} y$, we are led to the Hermitian modular group (cf. \cite{KRaW}).\\[1ex]
 c) Considering the Hurwitz quaternions or the order of integral Cayley numbers, confer \cite{H-WK} and \cite{DKW}. These cases refer to the root lattices $D_4$ and $E_8$ (cf. sect. 4).
 \end{example}
 
 The case of non-maximal lattices is dealt with in the following Remark.
 
 \begin{remark}\label{Remark_1} 
 An arbitrary even lattice $L$ is contained in a maximal even lattice $L^*$ with Gram matrix $S^*$. Hence there exists a matrix $H\in \Z^{n\times n}$ satisfying
\[
 S = H^{tr} S^* H, \;\; S_1 = \widehat{H}^{tr} S^*_1 \widehat{H},\;\; \widehat{H} = 
 \begin{pmatrix}
  I & 0 & 0 \\ 0 & H & 0 \\ 0 & 0 & I
 \end{pmatrix}.
\]
Clearly $|\det H| = [L^*:L]$ holds. In this case we have
 \[
  \widetilde{\Gamma}_S \subseteq \widehat{H}^{-1} \widetilde{\Gamma}_{S^*} \widehat{H} \subseteq \widehat{H}^{-1} \Gamma_{S^*} \widehat{H} \subseteq SO^+(S_1;\Q).
 \]
 Thus $L$ is maximal even, whenever $\det S$ is squarefree. If $n$ is odd, then $\det S$ and $\det S^*$ are even. Thus $L$ of odd rank is maximal even, whenever $(\det S)/2$ is squarefree.
\end{remark}

We give an application to right and double cosets, which is also needed in the attached Hecke theory.

\begin{theorem}\label{theorem_2} 
Let $L_1 = \Z^{n+4}$ satisfy \eqref{gl_1} and 
 \begin{gather*}\tag{6}\label{gl_6}
  R' = \tfrac{1}{\sqrt{r}} R \in SO^+(S_1;\R), \; r\in\N, \; R\in \Z^{(n+4)\times(n+4)}.
 \end{gather*}
If $L$ is maximal even or $r$ and $\det S$ are coprime, the following holds. \\[1ex]
a) The right coset $\widetilde{\Gamma}_S R$ contains a matrix
 \begin{gather*}\tag{7}\label{gl_7} 
  \begin{pmatrix}
   \alpha & \ast & \ast \\ 0 & \ast & \ast \\ 0 & 0 & \delta
  \end{pmatrix}, \;\; \alpha,\delta \in \N,\; \alpha \delta = r,
 \end{gather*}
where $\alpha$ is the $\gcd$ of the first column of $R$.  \\[1ex]
b) The double coset $\widetilde{\Gamma}_S R \widetilde{\Gamma}_S$ contains a matrix 
 \begin{gather*}\tag{8}\label{gl_8} 
   R^* =  \begin{pmatrix}
\alpha^* & 0 & 0 \\ 0 & K^* & 0 \\ 0 & 0 & \delta^*
  \end{pmatrix} \in\Z^{(n+4)\times(n+4)},  \;\; \alpha^*\in\N,\; \alpha^*\delta^* = r,\; \tfrac{1}{\alpha^*}R^*\in\Z^{(n+4)\times(n+4)},
 \end{gather*}
where $\alpha^*$ is the $\gcd$ of all the entries of $R$.
\end{theorem}

\begin{proof}
 a) Let $g$ be the first column of $R$, which satisfies $g^{tr} S_1 g = 0$ and let $\alpha = \gcd(S_1 g)$. \\[1ex]
 (i) ~If $L$ is maximal even, we have
 \[
  \tfrac{1}{\alpha} g\in S^{-1}_1 \Z^{n+4} = L^\sharp_1 \;\;\text{and}\;\; \tfrac{1}{\alpha} g\in L_1
 \]
due to Corollary \ref{corollary_1}. \\[1ex]
(ii) ~If $r$ and $\det S$ are coprime, we observe that $g^{tr}S_1$ is the last row of the matrix $\sqrt{r} R^{*-1}$ due to \eqref{gl_3}, which has the determinant $r^{2+n/2}$. As $\alpha$ divides a power of $r$ it is coprime to $\det S$ and we get again $\frac{1}{\alpha} g \in L_1$.   \\
As $\frac{1}{\alpha} g$ has coprime entries we conclude from Theorem \ref{theorem_1} that it is the first column of a matrix $M\in \widetilde{\Gamma}_S$ in both cases. Therefore $(\alpha,0,\ldots,0)^{tr}$ is the first column of $M^{-1} R$. \\
As block diagonal matrices form a subgroup, we obtain \eqref{gl_7} from the description of the inverse in \eqref{gl_3}. \\[1ex]
b) Let $\alpha^*$ be the smallest positive $(1,1)$-entry in $\Z$ of all the matrices in $\widetilde{\Gamma}_S R \widetilde{\Gamma}_S$ and assume 
\[
 R = \begin{pmatrix}
      \alpha^* & \ast \\ \ast & \ast
     \end{pmatrix}
\]
without restriction. It follows from a) that $\alpha^*$ divides the entries of the first column of $R$. The same procedure as in a) applied to $M^{-1}J$ shows that $\alpha^*$ also divides the entries of the first row of $R S^{-1}_1$. Multiplication by $T^*_\lambda$, $\lambda\in\Z^{n+2}$, from the left and by $T_\mu$, $\mu\in\Z^{n+2}$, from the right (cf. \eqref{gl_5}) leads to 
\[
  R^* = \begin{pmatrix}
   \alpha^* & 0 & 0 \\ 0 & K^* & 0 \\ 0 & 0 & \delta^*
  \end{pmatrix} \in \widetilde{\Gamma}_S R \widetilde{\Gamma}_S, \;\; \alpha^* \delta^* = r.
\]
Considering $R^* JT_\lambda $, $\lambda=e_1,\ldots, e_{n+2},e_1+e_{n+2} \in \Z^{n+2}$, then shows that $\alpha^*$ divides the entries of $K^*$ and $\delta^*$. Now $\widetilde{\Gamma}_S \subseteq SL(n+4;\Z)$ implies that $\alpha^*$ is the $\gcd$ of the entries of $R$.
\end{proof}

\section{Maximal Discrete Subgroups}

We follow the procedure by Ramanathan \cite{Ra}.

\begin{lemma}\label{lemma_1} 
 Let $L_1 = \Z^{n+4}$ satisfy \eqref{gl_1}. Let $\Delta$ be a discrete subgroup of $SO^+(S_1;\R)$, which contains $\widetilde{\Gamma}_S$. Then the following holds  \\[1ex]
 a) $[\Delta:\widetilde{\Gamma}_S] = j\in\N$.  \\[1ex]
 b) Given $R\in \Delta$, there exists $r\in \N$ such that
 \[
  \sqrt{r} R \in \Z^{(n+4)\times (n+4)}.
 \]
\end{lemma}

\begin{proof}
 a) According to \cite{Bru2}, 4.10, the discriminant kernel $\widetilde{\Gamma}_S$ possesses a fundamental domain with respect to the action on the orthogonal half-space with finite, positive volume. As $\Delta$ is countable, the index $j$ must be finite.\\[1ex]
 b) Assume the notation \eqref{gl_2} for $R$. Multiplying $R$ by matrices of type \eqref{gl_5}, we may assume that $\alpha$, $\beta$, $\gamma$, $\delta$ are non-zero. Setting $k=j!$ we conclude
 \[
  R^k \in \widetilde{\Gamma}_S \;\; \text{for each}\;\; R\in\Delta.
 \]
This leads to
\[
 (RMR^{-1})^k = RM^k R^{-1} \in\widetilde{\Gamma}_S, \;\; \text{i.e.}\;\; R(M^k-I)R^{-1} \in \Z^{(n+4)\times(n+4)} S_1
\]
for all $M\in\widetilde{\Gamma}_S$. Using $M=T_\lambda$ in \eqref{gl_5} with $\lambda=(1,0,\ldots,0)^{tr}$ and $\mu = (0,\ldots,0,1)^{tr}$ we get 
\begin{align*}
 & RXR^{-1} \in \Z^{(n+4)\times(n+4)}S_1 \;\; \text{for}\;\; X=T_{k\lambda}+T_{k\mu}-T_{k(\lambda+\mu)}-I=k^2
 \begin{pmatrix}
  0 & 0 & 1 \\ 0 & 0 & 0  \\ 0 & 0 & 0
 \end{pmatrix},   \\
 & k^2 xx^{tr} \in \Z^{(n+4)\times (n+4)},
\end{align*}
whenever $x$ is the first column of $R$ or the last column of $R$, if we replace $T_\nu$ by $T^*_\nu$. Considering $R^{-1}$ instead of $R$ this remains true, if $x^{tr} S_1$ is the first or last row of $R$. Elementary number theory yields for $f,g\in\R$ with $f^2,g^2,fg\in\Z$ that $f,g\in\Z\sqrt{h}$ for some $h\in\N$.
If $k^2\alpha^2 = u^2 v$ with $u,v\in \N$, $v$ squarefree, we conclude that $\rho x$ is integral for all the vectors $x$ mentioned above, whenever
\begin{gather*}\tag{9}\label{gl_9} 
\rho = k/\sqrt{v}.
\end{gather*}
If we replace $R$ by $RT_\lambda$, $\lambda \in \Z^{n+2}$, we conclude that $\rho K$, hence $\rho R$ is integral, too. 
Then 
\[
 \rho^2S_1 = (\rho R)^{tr} S_1(\rho R) \in \Z^{(n+4)\times (n+4)}
\]
leads to $\rho^2\in \N$.
\end{proof}

Due to the determinantal condition, $R$ in Lemma \ref{lemma_1} is a rational matrix, whenever $n$ is odd.

\begin{corollary}\label{corollary_2} 
Let $L_1 = \Z^{n+4}$ satisfy \eqref{gl_1}. Then the normalizer of $\widetilde{\Gamma}_S$ in $SO^+(S_1;\R)$ is equal to $\Gamma_S$.
\end{corollary}

\begin{proof}
 Clearly $\widetilde{\Gamma}_S$ is a normal subgroup of $\Gamma_S$. Given $R$ in the mormalizer of $\widetilde{\Gamma}_S$ in 
 $SO^+(S_1;\R)$, we conclude from the proof of Lemma \ref{lemma_1} and in particular from \eqref{gl_9} with $k=1$ that $\frac{1}{\sqrt{v}} R$ is integral. Then $\det R = 1$ leads to $v=1$ and therefore $R\in\Gamma_S$.
\end{proof}

Next we consider the particular case of maximal even lattices. 
\begin{theorem}\label{theorem_3} 
 Let $L_1=\Z^{n+4}$ be a maximal even lattice satisfying \eqref{gl_1}. Then $\Gamma_S$ is the uniquely determined maximal discrete extension of $\widetilde{\Gamma}_S$ in $SO^+(S_1;\R)$ and coincides with the normalizer of $\widetilde{\Gamma}_S$ in $SO^+(S_1;\R)$.
\end{theorem}

\begin{proof}
Let $\Delta$ be a discrete subgroup of $SO^+(S_1;\R)$, which contains $\widetilde{\Gamma}_S$.
 Due to Lemma \ref{lemma_1} and Theorem \ref{theorem_2} we may assume 
 \[
  R = \begin{pmatrix}
       \alpha & 0 & 0 \\ 0 & K & 0 \\ 0 & 0 & \delta
      \end{pmatrix} \in\Delta, \;\;
  0<\alpha\leq \delta,\; \alpha\delta=1,\;\; \tfrac{1}{\alpha} R\in\Z^{(n+4)\times(n+4)}.
 \]
If $0<\alpha < 1$ the right cosets $\widetilde{\Gamma}_S R^m$, $m\in \Z$, are mutually different. This contradicts Lemma \ref{lemma_1}. Thus $\alpha = 1$ and $R$ is integral, i.e. $R\in\Gamma_S$. As $\Gamma_S$ is clearly a discrete group, it is the unique maximal discrete extension of $\widetilde{\Gamma}_S$ and coincides with the normalizer due to Corollary \ref{corollary_2}.
\end{proof}

Non-maximal lattices are described in the following Remark.

\begin{remark}\label{Remark_2}  
a) If $L$ is not maximal even, $r$ in Lemma \ref{lemma_1} is not always equal to $1$. But one can proceed along the proof of Lemma 4 in \cite{KRaW} in order to show that $r$ is always a divisor of $(\det S)^m$ for some $m\in \N$. Now consider Remark \ref{Remark_1}. 
Hence $\widehat{H}^{-1} \Gamma_{S^*} \widehat{H}$ is a maximal discrete extension of $\widetilde{\Gamma}_S$ due to Theorem \ref{theorem_3}, which is neither contained in $SL(n+4;\Z)$ nor in the normalizer of $\Gamma_S$ or $\widetilde{\Gamma}_S$ due to Corollary \ref{corollary_2}. Note that this maximal discrete extension does not contain $\Gamma_S$ in general.
As a maximal even overlattice is not unique in general (cf. \cite{Ni}), we conclude that a maximal discrete extension $\widetilde{\Gamma}_S$ is not uniquely determined in general. More precisely any maximal discrete extension is equal to the normalizer $\Gamma_S$, if and only if the underlying lattice $L$ is maximal even. \\[1ex]
b)  Lemma \ref{lemma_1} remains true, if we replace the discriminant kernel by an arbitrary congruence subgroup of $\Gamma_S$ .
\end{remark}

\section{Root Lattices}

In this section we deal with root lattices, as they yield the most common examples of Borcherds products (cf. \cite{WW2}). We quote \cite{CS}, Chap. 4, and \cite{Eb}, 1.4, for details.  

The lattice $A_n$ is given by 
\[
 A_n:= \bigl\{\lambda \in\Z^{n+1};\; \lambda_1+\ldots +\lambda_{n+1} = 0\bigr\}, \;\; \disc A_n=n+1.
\]
The discriminant group is cyclic of order $n+1$
\[
 A^\sharp_n/A_n = \langle a_n+A_n\rangle, \;\; a_n=\tfrac{1}{n+1}(ne_1-e_2-\ldots -e_{n+1}), \;\; \overline{q}(a_n+A_n) = \tfrac{n}{2(n+1)} +\Z.
\]
Hence $(A^\sharp_n/A_n,\overline{q})$ is anisotropic if and only if
\[
 8 \nmid(n+1) \;\;\text{and}\;\; p^2\nmid(n+1)\;\;\text{for each odd prime}\; p.
\]
If $A_n$ is not maximal, its maximal discrete extension is uniquely determined and generated by $A_n$ and $j a_n$ with $nj^2 \equiv 0 \pmod{2(n+1)}$. It is equal to 
\[
 \bigl\{\lambda \in A^\sharp_n;\; \langle \lambda,\lambda\rangle \in 2\Z\bigr\}.
\]
The lattice $D_n$ is given by
\[
 D_n = \bigl\{\lambda\in\Z^n;\; \lambda_1+\ldots +\lambda_n\equiv 0\!\!\!\!\pmod{2}\bigr\}, \;\; \disc D_n=4.
\]
If $n$ is odd, $D_n$ is maximal even due to Remark \ref{Remark_1}. If $n$ is even, $D^\sharp_n/D_n$ is a Kleinian $4$-group, where the values of $\overline{q}$ are given by
\[
 \Z,\;\; \tfrac{1}{2}+\Z,\;\; \tfrac{n}{8}+\Z,\;\; \tfrac{n}{8}+\Z.
\]
Hence $D_n$ is maximal even, if and only if $8\nmid n$. If $8 \mid n$, the lattice $D^+_n$ generated by $D_n$ and $h_n=\frac{1}{2}(e_1 +\ldots + e_n)$ is unimodular and a maximal even overlattice of $D_n$.
\medskip

$E_8 = D^+_8$ is unimodular. $E_7 = \langle e_7-e_8\rangle^\perp$ has $\disc E_7=2$. $E_6 =  \langle e_6-e_7,e_7-e_8\rangle^\perp$ satisfies $\disc E_6=3$. Hence they are maximal even due to Remark \ref{Remark_1}.
\medskip

Summarizing we have

\begin{lemma}\label{lemma_2}  
 A complete list of maximal even irreducible root lattices is given by \\[0.5ex]
a) $A_n$, if $n$ is even and $n+1$ is squarefree or \\
 \hspace*{7ex} if $n$ is odd and $(n+1)/2$ is squarefree.\\[0.5ex]
b) $D_n$, if $n$ is not a multiple of $8$.\\[0.5ex]
c) $E_6$, $E_7$, $E_8$.
\end{lemma}

Clearly one can deal with arbitrary root lattices on this basis. There exists a unique overlattice, which is maximal even, for instance, whenever the discriminant group is cyclic, as pointed out for $A_n$. But $D_n$ for $8 \mid n$ has got two different maximal even overlattices, which are both isometric to $D_n^+$.

\begin{example}\label{example_2} 
 a) Let $L=4 A_1$. Then $\Gamma_S$, $S=2I^{(4)}$, corresponds to the extended modular group over the Lipschitz quaternions and admits a unique maximal discrete extension, which is given by $L^*=D_4$ and corresponds to the extended modular group over the Hurwitz quaternions (cf. \cite{H-W}).  \\[1ex]
 b) Let $L=5 A_1$, $S=2I^{(5)}$. Then there are $5$ maximal even overlattices given by $\Z^5 + \Z h_j$, $h_j= \frac{1}{2}(h-e_j)$, $ j= 1,\ldots,5$, where $h= (1,1,1,1,1)^{tr}$. The associated maximal discrete extensions of $\widetilde{\Gamma}_S$ in Remark \ref{Remark_2} are isomorphic.  \\ 
 But there is another maximal discrete extension given by $\Gamma_S$, which is not isomorphic to the other ones. This can be proved in a similar way as in \cite{H-WK}, as one can restrict to matrices with denominator $2$ in the maximal discrete extension and uses the fact that the matrices 
 $\left(\begin{smallmatrix}
         I & 0 & 0 \\ 0 & P & 0 \\ 0 & 0 & I
        \end{smallmatrix}\right)$,
 $P\in SO(I^{(5)};\Z)$, belong to $\Gamma_S$.
\end{example}
\medskip

The authors would like to thank Gabriele Nebe for helpful discussions.

\nocite{WW1} \nocite{Wa}


\bibliography{bibliography_krieg_2021.bib} 
\bibliographystyle{plain}

\end{document}